\documentclass[11pt, reqno,fleqn]{amsart}

\usepackage{amsmath}
\usepackage{amssymb}
\usepackage{amsfonts}
\usepackage{graphicx}
\usepackage{amsthm}
\usepackage{enumerate}
\usepackage{lscape}
\usepackage{dsfont}
\usepackage{color}
\usepackage{mathtools}

\usepackage{geometry}
\usepackage{pstricks-add}
\usepackage{setspace}
\usepackage{cancel}

\numberwithin{equation}{section}

\newtheorem{theor}{Theorem}[section]
\newtheorem{prop}[theor]{Proposition}
\newtheorem{defin}[theor]{Definition}

\newtheorem*{postulate}{Postulate}
\newtheorem{lem}[theor]{Lemma}

\begin{document}

\title[ A different approach to teaching geometry in Italian secondary schools ]{A different approach to teaching geometry in Italian secondary schools}
\author[D. Uccheddu]{Daria Uccheddu}
\address{Dipartimento di Matematica e Informatica, Universit\`{a} di Cagliari,
Via Ospedale 72, 09124 Cagliari, Italy}
\email{ daria.uccheddu@unica.it }
\thanks{
The author was supported by KASBA-progetti ricerca di base - fcs annualit\`a 2017}

\date{}

\begin{abstract}
Inspired by a didactic experience in an academic environment, and following the idea given by M. Villa in \cite{Villa}, we illustrate two different proofs of an important result in Euclidean geometry studied in the first two years of Italian secondary schools. Specifically, we propose the proof of Varignon's Theorem from both the classical synthetic and the analytical points of view.
\end{abstract}

\maketitle

\section {Introduction} 
Teaching geometry in Italian schools is fundamental for the mathematical education of the students and, traditionally, proving a geometric theorem is a crucial activity in mathematical practice. Frequently, students in secondary schools consider the proof of a theorem a pointless thing, although it has a very important formative value in mathematical education (see \cite{Hanna} and \cite{Durand}).\\

The idea of this paper comes from a didactic experience about an open problem proposed during the course of Methodologies and Technologies for the didactics of Mathematics at  the University of Cagliari. Six student teachers participated (four females and two males), all of the first year of the Master Degree in Mathematics.
Even if the activity was born for stimulate students to conjecture a statement of the proposed problem using a dynamic geometry software (DGS), and, then provide its proof, (see \cite{Boero}, \cite{Mariotti} and \cite{Pedemonte}), the same activity showed that for almost all the students it was easier to provide an analytical proof rather than a synthetic one. \\

It is well known that teaching Euclidean geometry is of fundamental importance in Italian secondary schools, in all grades of studies. For example, the Cartesian plane as a tool for studying geometry is one of the aims of the first cycle of the secondary schools (see \cite{D.M 254 2012} p.51),  and of course Euclidean geometry is the most suitable tool for understanding the space around us.\\

It is therefore clear that it is essential tackling geometric problems using tools and techniques that can also be used in everyday experience. In fact, mathematical outgoing skills are acquired both by students who will continue with a graduate education and by students who will enter into the job market. 
All these elements lead us to consider a different approach to mathematics and in particular to geometry. 
At the same time, it is necessary not to reduce the importance of a rigorous and logical approach to the study of Euclidean geometry as presented in Euclid's Elements. Thus, Euclidean geometry's theorems should be proven twice by using metric tools and by  applying synthetic techniques.\\

As an example, in this paper we propose two proofs of Varignon's Theorem from the synthetical and the analytical points of view. The first proof is the original one given by Varignon in \cite{Varignon}, which uses synthetic geometry following Euclid's Elements (\cite{Euclid}). The second proof is constructed fixing a Cartesian frame and using analytic tools.\\
The paper is organized as follow.
In Section \ref{HF} we recall briefly the aspects of the Italian school system which are significant in our discussion. In Section \ref{EGCG} we explain definitions, propositions and all the geometrical objects we need. Finally,  in Section \ref{DE} we propose the two proofs of Varignon's Theorem.

\section{Some aspects of teaching geometry in the Italian School System}\label{HF}
Finding the best way of teaching geometry in Italian secondary schools has been historically and still is largely studied by teachers and researchers in mathematical education (see \cite{Herb}). 
As planned in the National Italian program, the students start learning geometry in the Elementary schools (I-V grades). 
In the first cycle of the secondary schools (VI-VIII grades), geometry is approached in an intuitive and empirical way (see \cite{Villa}). Classically, the last subject dealt with is solid geometry from a metric point of view, which consists in applying formulas to compute volumes and areas of classical geometric surfaces (see \cite{D.M 254 2012} p.52). In fact, the Italian National Program for I-VIII grades (\cite{D.M 254 2012}) is a collection of goals that students must achieve at the end of VIII grade. In particular, this program lists the objectives in four areas of mathematics, i.e.  ``Numbers", ``Space and Figure", ``Relations and Functions" and ``Data and Foresight".\\

In particular in the list of goals \cite[p.52]{D.M 254 2012} relative to ``Space and Figure", also called ``Geometry and Measuring",  we read: 

\begin{itemize} 
\item reproducing figures and geometric pictures, using suitable and appropriate tools (ruler,  set square, drawing compass, $180$ degree protractor, Geometric software).
\item drawing points, segments and figures in the Cartesian Plane.
\item learning definitions and properties (angles, line of symmetry, diagonals) of fundamental plane figures (triangles, quadrilaterals, regular polygons, circle).
\item computing area and volume of most common solid figures (...).
\end{itemize}
Thus, at the end of the VIII grade students acquire a concrete and closely related to the concept of measure idea of geometry.

On the other hand, in IX and X grades, Euclidian geometry is developed through an exclusively synthetic approach, following Euclid's Elements (see \cite{Euclid}), while in XI and XII grades students approach Euclidian geometry from an analytic point of view. \\

Analyzing the National Italian programs it thus seems that there is a gap between the first and the second cycle of secondary schools.\\
Taking into account that an axiomatic approach (in the strictly classical sense, see \cite{Villa} for more details) to Euclidian's geometry is of crucial importance for the intellectual and educational value
of mathematics (see \cite{Vinner}), we believe that one way to fill the gap could be applying both an analytic and a synthetic strategy in parallel, 
in order to give a logic justification to introduce the Cartesian plane and to use coordinates, already in the first cycle of the secondary school. \\

\section{Notions of Euclidean Geometry}\label{EGCG}
In this section we briefly list definitions, theorems and postulates that we will use in the next and  last section. \\
Let us first recall the fifth Euclid's postulate  and some basic definitions (we follow here the English translation of Euclid's Elements given in \cite{Euclid})\begin{postulate}[Parallel postulate or fifth Euclid's postulate]\label{5postulato}
For any line $r$ and point $P$ not on $r$, there exists a unique line through $P$ not meeting $r$.\end{postulate}

\begin{defin} [Parallel lines]
Parallel lines are straight-lines which, being in the same plane, and being produced to infinity in each direction, meet with one another in neither (of these directions).

\end{defin}
\begin{defin} Given two points $A$ and $B$ in a straight-line $r$ we call segment $AB$ the subset of points of $r$ between $A$ and $B$.
\end{defin}
We also need the following results:
 \begin{prop}
Straight-lines parallel to the same straight-line are also parallel to one another.

\end{prop}
\begin{proof}
See \cite{Euclid}, Prop. $30$ of Book $1$.
 \end{proof}
\begin{prop}[Basic Proportionality Theorem and converse]\label{CBOT}
If some straight-line is drawn parallel to one of the sides of a triangle then it will cut the (other) sides of the triangle proportionally. \\
And if (two of) the sides of a triangle are cut proportionally then the straight-line joining the cutting (points) will be parallel to the remaining side of the triangle.

\end{prop}
\begin{proof}
See \cite{Euclid}, Prop. $2$ of Book $6$.
\end{proof} 
The theorem we deal with in our didactic experience is Varignon's Theorem that, stated as appears in \cite{Varignon}, reads:
\begin{theor}[Varignon's Theorem]\label{Var}
The midpoints of the sides of an arbitrary quadrilateral form a parallelogram.
\end{theor}

In the proof of Varignon's Theorem we will also need:
\begin{defin}\label{midpoint}
Given a segment $AB$, the {\em midpoint} of $AB$ is the point $M$ equidistant from both $A$ and $B$, i.e. the point $M$ such that $AM$ is congruent to $MB$.
\end{defin}
Finally, we recall the definition of parallelogram:
\begin{defin}
A parallelogram is a quadrilateral with two pairs of parallel opposite sides.

\end{defin}
In order to give an analytic proof of Varignon's Theorem, we consider the plane as a 2-dimensional Euclidian space endowed with a metric structure, where we fix a Cartesian frame that allows us to describe points in coordinates. 
In this setting by Definition \ref{midpoint} we have:
\begin{lem}
Given a segment $AB$, with $A=(x_A,y_A)$ and $B=(x_B,y_B)$, the midpoint $M$ of $AB$ is the point of coordinates 
\begin{equation}\label{puntomedio}
M=\left(\frac{x_A+x_B}{2},\frac{y_A+y_B}{2}\right).
\end{equation}
\end{lem}
Finally, we conclude this section by recalling two properties of the slope of a straight-line:
\begin{lem}\label{slope}
Let  $A$ and $B$ be two points with $A=(x_A,y_A)$ and $B=(x_B,y_B)$, the slope of the straight-line that contains $A$ and $B$ is given by \begin{equation}\label{coef}
m=\frac{y_B-y_A}{x_B-x_A}.
\end{equation}
\end{lem}
\begin{lem}\label{parall}
Parallel lines have the same slope.
\end{lem}

\section{Didactic experience}\label{DE}
The idea of a different approach in teaching geometry in high school, using in parallel analytic and synthetic tools comes from a didactic experience achieved during the academic Course  ``tecnologie e metodologie didattiche per l'insegnamento della matematica". More precisely, we proposed the following activity: \\

\texttt{Given a quadrilateral $ABCD$, denote by $F$, $G$, $H$ and $E$ the midpoints of the sides $AB$, $BC$, $CD$, $DA$, respectively.\\
What kind of properties the quadrilateral $EFGH$ has?  Write a theorem that explains the figure you have constructed.}\\

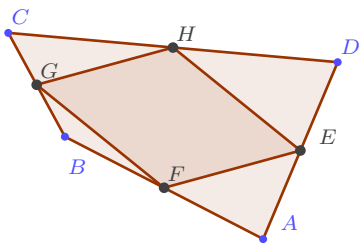
\newrgbcolor{ududff}{0.30196078431372547 0.30196078431372547 1}
\newrgbcolor{zzttqq}{0.6 0.2 0}
\psset{xunit=0.2cm,yunit=0.2cm,algebraic=true,dimen=middle,dotstyle=o,dotsize=3pt 0,linewidth=1pt,arrowsize=1pt 1,arrowinset=0.25}
\begin{pspicture}
(-44.747596158964015,-14.12972444951273)(25.146088588805537,6.075204913347687)
\psaxes[labelFontSize=\scriptstyle,xAxis=false,yAxis=false,Dx=5,Dy=5,ticksize=-2pt 0,subticks=2]{->}(0,0)(-34.747596158964015,-29.12972444951273)(25.146088588805537,16.075204913347687)
\pspolygon[linewidth=1pt,linecolor=zzttqq,fillcolor=zzttqq,fillstyle=solid,opacity=0.1](-0.6823954432763722,-11.303402605876652)(-13.72,-4.52)(-17.440058505530697,2.3627605466169896)(4.22,0.42)
\pspolygon[linewidth=1pt,linecolor=zzttqq,fillcolor=zzttqq,fillstyle=solid,opacity=0.1](-7.201197721638186,-7.911701302938326)(-15.58002925276535,-1.078619726691505)(-6.610029252765349,1.3913802733084948)(1.7688022783618138,-5.441701302938326)
\begin{scriptsize}
\psdots[dotstyle=*,linecolor=ududff](-0.6823954432763722,-11.303402605876652)
\rput[bl](0.4738335941757357,-10.740313638172786){\ududff{$A$}}
\psdots[dotstyle=*,linecolor=ududff](-13.72,-4.52)
\rput[bl](-13.511527327292873,-6.9652675497843863){\ududff{$B$}}
\psdots[dotstyle=*,linecolor=ududff](-17.440058505530697,2.3627605466169896)
\rput[bl](-17.2121827537235,2.923644859417264){\ududff{$C$}}
\psdots[dotstyle=*,linecolor=ududff](4.22,0.42)
\rput[bl](4.422418200794018,0.9879174055920068){\ududff{$D$}}
\psdots[dotsize=4pt 0,dotstyle=*,linecolor=darkgray](-7.201197721638186,-7.911701302938326)
\rput[bl](-6.964213880530991,-7.4381903345885245){\darkgray{$F$}}
\psdots[dotsize=4pt 0,dotstyle=*,linecolor=darkgray](-15.58002925276535,-1.078619726691505)
\rput[bl](-15.333388460304874,-0.606211085793499){\darkgray{$G$}}
\psdots[dotsize=4pt 0,dotstyle=*,linecolor=darkgray](-6.610029252765349,1.3913802733084948)
\rput[bl](-6.394882276464741,1.841914811691385){\darkgray{$H$}}
\psdots[dotsize=4pt 0,dotstyle=*,linecolor=darkgray](1.7688022783618138,-5.441701302938326)
\rput[bl](2.9742923033091415,-4.99006443710364){\darkgray{$E$}}
\end{scriptsize}
\end{pspicture}

We proposed, this activity,  given in Baccaglini-Frank (see \cite{Baccaglini}, p.174 ), to stimulate  students to make use of a DGS (for example Geogebra \cite{narh}) to give a solution of the ``open problem" proposed (see \cite{Boero}, \cite{Pedemonte} and \cite{Mariotti} for more examples).\\

The software shows the parallelism between $FG$ and $EH$ and between $FE$ and $GH$, so that one is {\bf leaded} to formulate Varignon's Theorem \ref{Var}. 
The Theorem and its proof were published postumo in $1731$ in the volume ``Elemens de Mathematique de Monsieur Varignon" (see \cite{Varignon}).\\

We quote the proof of Varignon's Theorem translated in \cite{oliver} from its French edition.
In the original work of Varignon \cite{Varignon}, the notation $AF\cdot  FB :: AE\cdot ED$ represents proportionality between segments, that is $AF$ is to $FB$ as $AE$ is to $ED$. Moreover, in the text ``Part $2$" correspond to second part of ``Proposition \ref{CBOT}".
\begin{proof}[Synthetic proof]
\begin{quotation}\small
If sides $AB$, $BC$, $CD$ and $DA$ of a quadrilateral are each divided into two equal parts by points $F$, $G$, $H$ and $E$, such that the midpoints are joined by segments $FE$, $EH$, $HG$, and $GF$, the quadrilateral figure $FEGH$ is a parallelogram; this may be shown by the extension of diagonals $DB$ and $AC$, since by hypothesis, $AF = FB$ and $AE = ED$, $AF\cdot  FB :: AE\cdot ED$, and thus (by Part~$2$) $EF$ is parallel to $DB$. In like manner, by hypothesis, $BG = GC$ and $DH = HC$; $BG \cdot GC :: DH \cdot HC$, and consequently (by Part~$2$) $GH$ will also be parallel to line $BD$. Since $EF$ and $GH$ are parallel to the same third line, they are therefore parallel between themselves.
By the same reasoning, one could prove that the lines $FG$ and $EH$ are parallel to the diagonal $AC$, and as a consequence, are parallel between themselves. Thus, the quadrilateral $EFGH$ is a parallelogram.
\end{quotation}
\end{proof}
In \cite{oliver} the author explains that in Varignon’s arrangement of the Euclidean geometric propositions, Part~$2$ of Theorem $29$ stated as ``the line that cuts a triangle proportionally is parallel to its base"(\cite{Varignon} p. $60$) corresponds to the second part of Proposition \ref{CBOT}.\\

This proof is the expected one for an Italian student of the first year in secondary schools. Although, the students involved in our experience rather constructed an analytic proof using a Cartesian frame and the formulas for the slope of a straight-line given in Lemmas \ref{slope} and \ref{parall}. Their proof is the following:

\begin{proof}[Analytic Proof]
Consider a Cartesian frame with origin $(0,0)$ and axes $x$ and $y$, then the points $A$, $B$, $C$ and $D$ in Cartesian coordinates are given by
\[
A=(x_A,y_A),\;\;\:\: B=(x_B,y_B),\;\; \:\:C=(x_C,y_C), \; \; \:\:D=(x_D,y_D)
\]
and, consequently, the midpoints are the following  
\[
\begin{aligned}
&F=\left(\frac{x_A+x_B}{2},\frac{y_A+y_B}{2}\right),\:\:\:\:\:\: G=\left(\frac{x_B+x_C}{2},\frac{y_B+y_C}{2}\right),\\
&H=\left(\frac{x_C+x_D}{2},\frac{y_C+y_D}{2}\right),\:\:\:\:\:\: E=\left(\frac{x_D+x_A}{2},\frac{y_D+y_A}{2}\right).\\
\end{aligned}
\]
Now, using formula \eqref{coef} for the slope of a straight-line containing two points, we have
\[\begin{aligned}
&m_{FG}=\frac{y_C-y_A}{x_C-x_A}, \:\:\:&m_{HE}=\frac{y_A-y_C}{x_A-x_C},\:\:\:\:
&m_{GH}=\frac{y_D-y_B}{x_D-x_B},\:\:\:\: &m_{FE}=\frac{y_B-y_D}{x_B-x_D},
\end{aligned}\]
and from Lemma \ref{slope} the sides $FG$ and $HE$ are parallel and also, by the same reason, $GH$ and $FE$ which implies that $EFGH$ is a parallelogram.
\end{proof}

\end{document}